\documentclass[12pt]{amsart}
\usepackage{amssymb, latexsym}
\usepackage{amsthm}
\usepackage{amsaddr}

\theoremstyle{plain}
\newtheorem{thm}{Theorem}[section]
\newtheorem*{introthm}{Theorem}
\newtheorem{cor}[thm]{Corollary}
\newtheorem{prop}[thm]{Proposition}
\newtheorem{lmm}[thm]{Lemma}
\newtheorem{conj}[thm]{Conjecture}

\theoremstyle{definition}
\newtheorem{definition}[thm]{Definition}
\newtheorem{remark}[thm]{Remark}
\newtheorem{example}[thm]{Example}

\numberwithin{equation}{section}

\usepackage[hypertex]{hyperref}

\usepackage[dvipsnames]{color}

\begin{document}

\title[MTS and the Eisenstein Integers]{Distributive Mendelsohn triple systems and the Eisenstein integers}
\author[A. W. Nowak]{Alex W. Nowak}
\address{Department of Mathematics\\
Iowa State University\\
Ames, Iowa 50011, U.S.A.}
\email{anowak@iastate.edu}

\keywords{quasigroup, Mendelsohn triple system, distributive, Eisenstein integers, Moufang loop, congruence subgroup, orthogonal Latin square}
\subjclass[2010]{20N05, 05B07}

\begin{abstract}
We define a Mendelsohn triple system (MTS) with self-distributive quasigroup multiplication and order coprime with $3$ to be distributive, non-ramified (DNR).  We classify, up to isomorphism, all DNR MTS and enumerate isomorphism classes (extending the work of Donovan, Griggs, McCourt, Opr\v{s}al, and Stanovsk\'{y}).  The classification is accomplished via the representation theory of the Eisenstein integers, $\mathbb{Z}[\zeta]=\mathbb{Z}[X]/(X^2-X+1)$.  Containing the class of DNR MTS is that of MTS with an entropic (linear over an abelian group) quasigroup operation.   Partial results on the classification of entropic MTS with order divisible by $3$ are given, and a complete classification is conjectured.  We also prove that for any entropic MTS, the qualities of being non-ramified, pure, and self-orthogonal are equivalent.  We introduce the varieties $\mathbf{RE}$ and $\mathbf{LE}$ of (resp. right and left) Eisenstein quasigroups, whose respective linear representation theories correspond to the alternative presentation $\mathbb{Z}[X]/(X^2+X+1)$ of the Eisenstein integers.
\end{abstract}

\maketitle

\section{Introduction}
The \emph{left self-distributive} (LD) rule
\begin{equation}\label{Eq:TheLeftDistRule}
    x(yz)=(xy)(xz)
\end{equation}
is a well-studied phenomenon appearing in the theory of knots \cite{Joyce}, Hopf algebras \cite{Racks}, symmetric spaces \cite{SymmSpaces}, term rewriting systems \cite[Ch.~V]{Dehernoy}, quasigroups \cite{Stanovsky}, and combinatorial designs \cite{Donovan}.  The results of this paper lie at the intersection of the latter two subjects, and what is more, bring algebraic number theory into the fold of disciplines connected by self-distributivity. 

\indent In varieties (we use variety in the universal algebraic sense: a class of algebras defined by identities) of idempotent algebras, wherein  

\begin{equation}\label{Eq:Idempotence}
    x^2=x
\end{equation}
holds, \eqref{Eq:TheLeftDistRule} follows from the \emph{entropic} (sometimes referred to as \emph{medial}) identity

\begin{equation}\label{Eq:Medial}
    (xy)(zu)=(xz)(yu).
\end{equation}
The \emph{right self-distributive} (RD) law
\begin{equation}\label{Eq:RDistRule}
    (yz)x=(yx)(zx)
\end{equation}
also follows from \eqref{Eq:Idempotence} and \eqref{Eq:Medial}.  Quasigroups in which both   \eqref{Eq:TheLeftDistRule} and \eqref{Eq:RDistRule} hold are referred to as \emph{distributive quasigroups}.  Idempotent, entropic quasigroups are distributive.  We shall concern ourselves with distributivity in the variety of \emph{Mendelsohn quasigroups}.  Mendelsohn quasigroups realize Mendelsohn triple systems as algebraic objects.  The variety is specified by, in addition to quasigroup axioms (v.i. Section \ref{SubSec:QgpsMTS}), idempotence \eqref{Eq:Idempotence} and the \emph{semisymmetric} identity 

\begin{equation}\label{Eq:SemiSymm}
    y(xy)=x.
\end{equation}
\noindent A Mendelsohn quasigroup is left distributive if and only if it is right distributive, and an entropic Mendelsohn quasigroup is distributive.

Distributive quasigroups admit linear representations over (in general, nonassociative) commutative Moufang loops (v.i. Definition \ref{Def:LinQgp}), while entropic quasigroups are linear over abelian groups.  Section \ref{Sec:Background} is, in part, devoted to fleshing out the details of linear quasigroup theory.  A result due to Fischer \cite{Fischer}, Galkin \cite{Galkin}, and Smith \cite{FDQ} (reproduced below as Theorem~\ref{Th:FSG}) breaks up the study of finite, distributive quasigroups into two cases: those which are linear over abelian groups of prime power order (the entropic case), and those which are linear over commutative Moufang loops of order $3^n$ for $n\geq4$ (distributive, nonentropic).  This result allows us to restrict our attention to abelian groups, and, in turn, to the realm of entropic quasigroups, and yet still obtain a classification of all distributive Mendelsohn quasigroups of order coprime to $3$.  This, our main result is a direct product decomposition theorem.  A (due to the absence of certain technicalities, slightly weaker) statement of this theorem is the following: 

\begin{introthm}\label{IntroThm:MainResult1}
Every distributive Mendelsohn triple system of order coprime with $3$ is isomorphic to a direct product of systems defined on abelian groups of the form $(\mathbb{Z}/_{p^n})^i$, for $i\in \{1, 2\}.$
\end{introthm}

Entropic, Mendelsohn quasigroups are ``essentially" --in a precise, categorical sense (v.i. Theorem \ref{Th:CatRelat})-- the same as finite modules over the Eisenstein integers, the Euclidean domain consisting of the integral elements of $\mathbb{Q}(\sqrt{-3})$.  Section \ref{Sec:Background} surveys the algebraic structure of the Eisenstein integers.  In Section \ref{Sec:BMM}, we will have enough context to explain our dubbing distributive MTS of order not divisible by $3$ \emph{distributive, non-ramified}.  We then go on to exploit Eisenstein module structure (that of modules over a PID) to obtain the aforementioned main result and enumerate the isomorphism classes described by this theorem.  Section \ref{Sec:Ramified} presents progress on, and complications related to, the case of $3$-divisibility.  In section \ref{Sec:IncProp}, we further leverage our direct product theorems to better understand combinatorial properties (purity, self-orthogonality, self-conversity) in distributive MTS.  We obtain the following characterization of DNR MTS:

\begin{introthm}\label{IntroThm:MainResult2}
An entropic Mendelsohn triple system has order coprime with $3$ if and only if it is pure if and only if it is self-converse (orthogonal to its converse).
\end{introthm}

\noindent Section \ref{Sec:RL} concludes the paper by considering varieties of quasigroups that correspond to an alternative presentation for the Eisenstein integers.

\section{Background}\label{Sec:Background}
\subsection{Quasigroups and Mendelsohn triple systems}\label{SubSec:QgpsMTS}
\indent A \emph{quasigroup} \\$(Q, \cdot, /, \backslash)$ is a (potentially empty) set $Q$ equipped with three binary operations: $\cdot$ denoting \emph{multiplication}, $/$ \emph{right division}, and $\backslash$ \emph{left division}; these operations satisfy the identities
\begin{equation}
\label{IL}
y\backslash(y\cdot x)=x;
\end{equation}
\begin{equation}
    \label{SL}
    y\cdot(y\backslash x)=x; 
\end{equation}
\begin{equation}
\label{IR}
(x\cdot y)/y=x;
\end{equation}
\begin{equation}
\label{SR}
(x/y)\cdot y=x.
\end{equation}

\noindent A \emph{homomorphism} between quasigroups is a function that respects multiplication and the two divisions.  These homomorphisms furnish a category $\mathbf{Q}$. 

\begin{remark}\label{Rmk:ConcatNotation}
When appropriate, we convey multiplication of quasigroup elements by concatenation, and in some instances, use both $\cdot$ and concatenation, our convention being that concatenation takes precedence over $\cdot$.  For example, $xy\cdot z\equiv(xy)z\equiv(x\cdot y)\cdot z.$
\end{remark} 

The variety of Mendelsohn quasigroups, denoted $\mathbf{MTS}$, is generated by \eqref{Eq:Idempotence}, \eqref{Eq:SemiSymm}, and \eqref{IL}-\eqref{SR}.  The homomorphisms between the objects in this variety form a category, which we also denote by $\mathbf{MTS}$.

\indent A finite Mendelsohn quasigroup $Q$ gives rise to a \emph{Mendelsohn triple system} $(Q, \mathcal{B}),$ where $\mathcal{B}$ is a collection of \emph{blocks}, or cyclic triples of \emph{points} from $Q$, satisfying the property that for any ordered pair of distinct points $(x, y)\in Q^2\setminus \widehat{Q}$, there is a unique block $(x\phantom{.}y\phantom{.}xy)\equiv(xy\phantom{.}x\phantom{.}y)\equiv(y\phantom{.}xy\phantom{.}x)\in \mathcal{B}$.  Conversely, any Mendelsohn triple system $(Q, \mathcal{B})$ corresponds to a finite Mendelsohn quasigroup on $Q$: let $x^2=x$ for any $x\in Q$, and $xy$ be the element in the block containing $(x, y)$ when $x\neq y$ (cf. \cite[Th.~1]{MTS}).  Divisions in the Mendelsohn quasigroup are given by the opposite of multiplication: $x/y=x\backslash y=yx$.  That is, every Mendelsohn quasigroup takes the form $(Q, \cdot, \circ, \circ)$, where $x\circ y=y\cdot x$.

\begin{remark}\label{Rmk:QgpvsMTS}
As the preceding lines indicate, Mendelsohn triple systems are conventionally restricted to nonempty, finite point sets of cardinality greater than $1$.  However, both the empty set and the trivial quasigroup (containing one element) are Mendelsohn quasigroups.  Moreover, there are infinite Mendelsohn quasigroups; take $\mathbb{C}$ under the multiplication $$xy=xe^{\pi i/3}+ye^{-\pi i/3}.$$ We do not directly deal with any infinite Mendelsohn quasigroups in this paper, and the empty and trivial quasigroups are, as Mendelsohn quasigroups, classified, so in order to employ the convenient acronym MTS, we will use the terms Mendelsohn quasigroup and Mendelsohn triple system interchangeably.  Any reference to isomorphism of MTS should be understood as an appeal to the quasigroup structure.   
\end{remark}

\subsection{Linearity and distributivity in quasigroups}\label{Sec:LinDistQgps} 
This section recasts \cite[Sec.~1-2.1]{Donovan}.  In the course of exposition, we will demonstrate the connection between linearity, entropicity, and distributivity in quasigroups.

\begin{remark}\label{Rmk:AlgNotation}
Throughout the paper, arguments will appear to the left of their functions.  For example, the image of a point $x$ under the function $f$ is written $xf,$ or even as $x^f$.  Exponentiation of an argument by its function can help to curb the proliferation of brackets to which works in nonassociative algebra are susceptible.       
\end{remark}

Let $(M, +, /, \backslash, 0)$ be a \emph{loop} --—a quasigroup with a multiplicative identity—-- where $0$ denotes the neutral element.  Assume, furthermore, that the multiplication is commutative.  If, for all $x, y, z\in M,$ 
\begin{equation}\label{Eq:CML}
    (x+x)+(y+z)=(x+y)+(x+z),
\end{equation}
then $(M, +, /, \backslash, 0)$ is a \emph{commutative Moufang loop} (CML).  The \emph{nucleus} of a CML $M$,
$$N(M)=\{x\in M\mid \forall y, z\in M\phantom{.} (x+y)+z=x+(y+z)\},$$ is the set of elements that associate with all of $M$.  An automorphism $R$ of a CML $M$ is \emph{nuclear} if $\{x+xR\mid x\in M\}\subseteq N(M)$.  Notice that any abelian group is a CML, and any CML-automorphism is also an abelian group automorphism, and is therefore nuclear.

\indent Let $\top$ denote the \emph{trivial quasigroup}, containing a single element. A quasigroup $Q$ with a pointed idempotent selected by a $\mathbf{Q}$-morphism $\top \to Q$ is a \emph{pique}, short for \emph{pointed, idempotent quasigroup}.

\begin{definition}\label{Def:LinQgp}
Let $(M, +)$ be a commutative Moufang loop.  Fix nuclear CML-automorphisms $R, L\in \text{Aut}(M)$.  The operation
\begin{equation}\label{Eq:LinQgp}
    xy=xR+yL
\end{equation}
endows $M$ with the structure of a quasigroup, $\text{Lin}(M, R, L)$.  Moreover, the map $\top\to \text{Lin}(M, R, L)$ selecting $0\in M$ gives $\text{Lin}(M, R, L)$ the structure of a pique, which is referred to as a \emph{CML-linear pique} on $(M, +)$.  If $(M, +)$ is an abelian group, then we simply say that $\text{Lin}(M, R, L)$ is a \emph{linear pique} on $(M, +)$.  
\end{definition}

\begin{remark}\label{Rmk:LinearAbbrevs}
($\mathrm{a}$) We adopt the convention that $\text{Lin}(M, R)$ denotes a CML-linear pique in which $L=1-R.$  This is because a CML-linear pique is idempotent if and only if $L=1-R$, and, until Section \ref{Sec:RL}, we are concerned only with Mendelsohn quasigroups.   

\vspace{3mm}

($\mathrm{b}$) When referring to a linear pique $\text{Lin}(M, R)$ that is Mendelsohn, we will often employ the phrase \emph{linear MTS}.

\vspace{3mm}

($\mathrm{c}$) Readers familiar with the quasigroup literature may notice that we have avoided the term \emph{affine quasigroup}, one more frequently used to describe the structures of Definition \ref{Def:LinQgp}.
However, since our work deals entirely with pointed structures, we felt it best to avoid the term affine and any origin-free geometric connotations that go along with it.          
\end{remark}

\indent Any CML-linear pique of the form $\text{Lin}(M, R)$ is idempotent and distributive.  The Belousov theorem (v.i. Th. \ref{Th:B&S}) establishes the converse for nonempty quasigroups. 

\indent  What are the conditions necessary for a CML-linear pique $\text{Lin}(M, R)$ to be Mendelsohn?  Well, as we have already noted in Remark \ref{Rmk:LinearAbbrevs}$(\mathrm b)$, such a structure is automatically idempotent.  Furthermore, semisymmetry \eqref{Eq:SemiSymm} forces
\begin{align*}
    xy\cdot x&=xR^2+yR-yR^2+x-xR \\
    &=y,
\end{align*}
and setting $y=0$ reveals $R^2-R+1=0$.  Conversely, any CML $M$ with nuclear automorphism $R$ satisfying $R^2-R+1=0$, or, equivalently, $1-R=R^{-1}$ furnishes an idempotent, semisymmetric quasigroup $\text{Lin}(M, R)$ (cf.\cite[Prop.~2.2]{Donovan}).

\begin{remark}\label{Rmk:HistInterlude}
Sade appears to be the first to have recognized the identity \\$X^2-X+1$ in relation to the construction of semisymmetric, idempotent quasigroups \cite{Sade2}.
In Mendelsohn's presentation of his eponymous triple systems, he used roots of this quadratic over finite fields to prove the existence of simple systems (i.e., MTS with no subsystems), and systems with no ``repeated blocks" (v.i. \emph{pure} MTS in Section \ref{Sec:Purity}) \cite[Ths.~6,~7]{MTS}.
\end{remark}

\begin{example}\label{Ex:MTS4}
Since $X^2-X+1\equiv X^2+X+1$ splits in $\mathsf{GF}(2^2),$ we can choose a root $\zeta$, and define a Mendelsohn quasigroup $\text{Lin}(\mathsf{GF}(2^2), \zeta)$ via \begin{equation*}
    xy=x\zeta+y(1-\zeta).
\end{equation*}
An equivalent matrix characterization of this linear MTS is afforded by $\text{Lin}((\mathbb{Z}/_2)^2, T),$ where $$T=\left(\begin{array}{cc}
    0 & 1 \\
    1 & 1
\end{array}\right).$$
\end{example}

\indent We review some fundamental results on the classification of entropic and distributive quasigroups.  With them, we will be able to clearly state the classification program initiated by Donovan et. al.

\begin{thm}[Bruck-Murdoch-Toyoda \cite{Bruck}, \cite{Murdoch}, \cite{Toyoda}]\label{Thm:TMB}
A nonempty quasigroup $(M, \cdot)$ is entropic if and only if there is an abelian group $(M, +)$ affording a linear pique representation $\text{Lin}(M, R, L)$ of the multiplication $\cdot$, and $RL=LR$.
\end{thm}

\begin{thm}[Belousov \cite{Belousov}]\label{Th:B&S}
A nonempty quasigroup $(M, \cdot)$ is distributive if and only there is a commutative Moufang loop $(M, +)$ affording a (idempotent) linear pique representation $\text{Lin}(M, R)$ of the multiplication $\cdot$.
\end{thm}

\begin{thm}[Fischer-Galkin-Smith \cite{Fischer}, \cite{Galkin}, \cite{FDQ}]\label{Th:FSG}
Let $n=p_1^{r_1}\cdots p_k^{r_k}$, where the $p_1, \dots, p_k$ represent pairwise distinct primes.  Suppose $Q$ is a distributive quasigroup of order $n$.  Then $Q$ is isomorphic to a direct product of distributive quasigroups $Q_1\times\cdots\times Q_k$ so that $|Q_i|=p_i^{r_i}.$  Moreover, if $Q_i$ is not a linear pique, then $p_i=3,$ $r_i\geq 4$, and it is CML-linear.
\end{thm}

\begin{remark}\label{Rmk:LinMTSSyn}
By Theorems \ref{Thm:TMB} and \ref{Th:B&S}, the reader should regard the phrases ``Mendelsohn linear pique," ``linear MTS," and ``nonempty, entropic Mendelsohn quasigroup" as synonyms.  
\end{remark}

In conjunction with the Chinese Remainder Theorem, Theorems \ref{Th:B&S} and \ref{Th:FSG} tell us that in order to classify distributive Mendelsohn triple systems of arbitrary order, it is enough to classify semisymmetric, idempotent linear piques on abelian groups of prime power order and semisymmstric, idempotent CML-linear piques over CML of exponent $3$, with the former class comprising all entropic MTS.  The following theorem \cite{Kepka} of Kepka and N\v{e}mec characterizes isomorphisms of idempotent CML-linear piques.  

\begin{thm}[Kepka-N\v{e}mec \cite{Kepka}]\label{Th:K&N}
Let $M_1$ and $M_2$ be commutative Moufang loops, and suppose $R_1$ is a nuclear automorphism of $M_1$ and $R_2$ a nuclear automorphism of $M_2$ such that $1-R_1$ and $1-R_2$ are automorphisms of $M_1$ and $M_2,$ respectively.  Then $\text{Lin}(M_1, R_1)\cong \text{Lin}(M_2, R_2)$ if and only if there is a loop isomorphism $\psi:M_1\to M_2$ such that $\psi^{-1}R_1\psi=R_2.$ 
\end{thm}

\begin{example}\label{Ex:NonisomAffQgps}
Both $3$ and $5$ are roots of $X^2-X+1 \mod 7$.  Thus, right multiplications by these field elements yield linear MTS structures $\text{Lin}(\mathbb{Z}/_{7}, 3)$, $\text{Lin}(\mathbb{Z}/_{7}, 3)$.  Because $\text{Aut}(\mathbb{Z}/_7)\cong C_6$ is commutative, Theorem \ref{Th:K&N} tells us $\text{Lin}(\mathbb{Z}/_7, 3)\ncong \text{Lin}(\mathbb{Z}/_7, 5)$.  
\end{example}

Now, we may understand the classification problem as determining the conjugacy classes of CML automorphisms (for abelian groups of prime power order and nonassociative CML of exponent $3$) which are annihilated by $X^2-X+1$.

In \cite{Donovan}, the authors classify distributive Mendelsohn quasigroups of orders $p$ and $p^2$, for any prime $p$.  They also enumerate, using GAP \cite{GAP}, isomorphism classes for prime powers less than $1000.$  The following theorem is the key to our improvement of these results.  We employ the the presentation $\mathbb{Z}[X]/(X^2-X+1)$ for the \emph{Eisenstein integers}.  We denote them by $\mathbb{Z}[\zeta]$, where $\zeta=e^{\pi i/3}=\frac{1}{2}+\frac{\sqrt{3}}{2}i$.  Algebraically, they may be specified as the rank-$2$ free abelian group with basis $\{1, \zeta\}$, inheriting $\mathbb{Z}$-algebra structure from complex multiplication.

\indent As a final matter of notation for the upcoming theorem, let us establish that for a category with finite products $\mathbf{V}$, we have $\mathbb{Z}\otimes\mathbf{V}$ standing for the subcategory of abelian group objects in $\mathbf{V}$.  Whenever $\mathbf{V}$ is a variety of quasigroups, the category $\mathbb{Z}\otimes\mathbf{V}$ is precisely the class of (homomorphisms between) linear piques in $\mathbf{V}$ (cf. \cite[Th.~10.4]{IQTR}).

\newpage 
\begin{lmm}
\label{Lmm:rootschar3}
Let $S$ be a nonzero (not necessarily commutative), unital ring, $f(X)=X^2-X+1\in S[X]$, and $\zeta\in S$ be a root of $f(X)$.
\begin{itemize}
    \item[(a)] The root $\zeta$ is a unit in $S$, and its inverse is also a root of $f(X)$. 
    \item[(b)] The identity $\zeta=\zeta^{-1}$ holds if and only if $S$ is a ring of characteristic $3$, and in this case, $\zeta=-1.$
\end{itemize}
\end{lmm}

\begin{proof}
(a)
Note $f(\zeta)=0$ demands $\zeta(1-\zeta)=1.$  Thus, $f(\zeta^{-1})=f(1-\zeta)=(1-\zeta)^2-1+\zeta+1=1-2\zeta+\zeta^2+\zeta=1-\zeta+\zeta^2=f(\zeta)=0.$  

\vskip 3mm
\noindent 
(b)
\indent If $\zeta=\zeta^{-1}=1-\zeta$, then  $1=2\zeta$.  This clearly cannot be if $\text{char}(S)=2,$ so proceed assuming $2\neq 0,$ and $\zeta^{-1}=2=\zeta$.  Then $2^2-2+1=0,$ equivalent to $3=0$.  Conversely, if $\text{char}(S)=3,$ then $f(X)=(X+1)^2.$ 
\end{proof}

\begin{thm}\label{Th:CatRelat}
Let $\mathbb{Z}\otimes\mathbf{MTS}$ denote the category of linear MTS, and $\mathbb{Z}[\zeta]$-$\mathbf{Mod}$ that of modules over the Eisenstein integers.  There is a faithful, dense functor $$F:\mathbb{Z}[\zeta]\text{-}\mathbf{Mod}\to\mathbb{Z}\otimes\mathbf{MTS}.$$    
\end{thm}

\begin{proof}
Note that a $\mathbb{Z}[\zeta]$-module is specified by a pair $(M, R),$ where $M$ is an abelian group, and $R$, an automorphism of $M$ annihilated by the polynomial $X^2-X+1$, represents the action of $\zeta$.  That $R$ is an automorphism follows from Lemma \ref{Lmm:rootschar3} applied to $S=\text{End}(M)$.  This observation permits a map $F:(M, R)\mapsto \text{Lin}(M, R)$ between objects.  If we can show that every $\mathbb{Z}[\zeta]$-module homomorphism is a homomorphism of quasigroups, then $F:f\mapsto f$ is a faithful, functorial correspondence between morphisms.  Indeed, given a module homomorphism $f:(M, R)\to (N, T)$, we find $(xR)f=(x^\zeta)f=(xf)^\zeta=(xf)T$.  Thus,

\begin{align*}
    (xy)f&=(xR+y(1-R))f\\
    &=(xR)f+yf-(yR)f\\
    &=(xf)T+yf-(yf)T\\
    &=(xf)T+(yf)(1-T)\\
    &=(xf)(yf).
\end{align*}
Now, if $\text{Lin}(M, R)$ is Mendelsohn, then $(M, R)$ is a $\mathbb{Z}[\zeta]$-module.  This confirms the density of $F$.
\end{proof}

\subsection{The Eisenstein integers}\label{SubS:EisenInt}
The goal of this section is to recall some basic structure of the Eisenstein integers.

\begin{remark}\label{Rmk:EisIntNotation}
In most of the number theory literature, (cf.  \cite{Hardy, Ireland}, for example) the Eisenstein integers are presented as $\mathbb{Z}[\omega]=\mathbb{Z}[X]/(X^2+X+1),$ using $\{1, e^{2\pi i/3}\}$ as an integral basis.  However, the map 
\begin{equation}
    \frac{1}{2}+\frac{\sqrt{3}}{2}i\mapsto \left(\frac{1}{2}+\frac{\sqrt{3}}{2}i\right)-1=\frac{-1}{2}+\frac{\sqrt{3}}{2}i 
\end{equation}
is a ring isomorphism mapping $\mathbb{Z}[e^{\pi i/3}]$ onto $\mathbb{Z}[e^{2\pi i/3}]$.    
\end{remark}

The Eisenstein integers have a faithful representation in $M_2(\mathbb{Z})$ given by 

\begin{equation}\label{Eq:EisIntRep}
    \zeta\mapsto \left(\begin{array}{cc}
        0 & -1 \\
        1 & 1
    \end{array}\right),
\end{equation}
and thus may be viewed as the $\mathbb{Z}$-algebra of integral matrices of the form 

\begin{equation}\label{Eq:EisIntMatRing}
    \left(\begin{array}{cc}
        a & 0 \\
        0 & a
    \end{array}\right)+\left(\begin{array}{cc}
        0 &-b  \\
        b & b
    \end{array}\right)=\left(\begin{array}{cc}
        a & -b \\
        b & a+b
    \end{array}\right).
\end{equation}
In fact,
\begin{equation}\label{Eq:EisIntNorm}
   (a+b\zeta)N :=\det \left( \begin{array}{cc}
        a & -b \\
        b & a+b
    \end{array}\right)=a^2+ab+b^2
\end{equation}
gives $\mathbb{Z}[\zeta]$ the structure of a Euclidean domain \cite[Prop.~1.4.2]{Ireland}.  
Notice, then, that $\mathbb{Z}[\zeta]$ is a PID, ensuring complete decomposability of any finitely generated module into a direct sum of cyclic modules: 

\begin{thm}\label{Th:StructThmfgEisMods}
Let $(M, R)$ be a finitely generated $\mathbb{Z}[\zeta]$-module.  Then there is a finite sequence of prime powers (which may include $0$) $\pi_1^{n_1}, \dots, \pi_l^{n_l}\in \mathbb{Z}[\zeta]$ so that  

\begin{equation}\label{Eq:StructThmfgEisMods}
(M, R)\cong \bigoplus_{i=1}^l (\mathbb{Z}[\zeta]/(\pi_i^{n_i}), R_i)    
\end{equation}
where each $R_i$ is an abelian group automorphism on $\mathbb{Z}[\zeta]/(\pi_i^{n_i})$. 
\end{thm}

\begin{cor}\label{Cor:ModDecompMTS}
Let $M$ be a finite abelian group, and $\emph{Lin}(M, R)$ a linear MTS.  Then there is a finite sequence of prime powers $\pi_1^{n_1}, \dots, \pi_l^{n_l}\in \mathbb{Z}[\zeta]$ so that 
\begin{equation}\label{Eq:ModDecompMTS}
\emph{Lin}(M, R)\cong \prod_{i=1}^{l} \emph{Lin}\left(\mathbb{Z}[\zeta]/(\pi_i^{n_i}), R_i\right),   
\end{equation}
where each $R_i$ is an abelian group automorphism on $\mathbb{Z}[\zeta]/(\pi_i^{n_i})$.
\end{cor}

\begin{proof}
The arbitrary linear MTS $\text{Lin}(M, R)$ is the image of the finite $\mathbb{Z}[\zeta]$-module $(M, R)$ under the functor $F$ of Theorem \ref{Th:CatRelat}.  Since $F:f\mapsto f$ injects morphisms, the isomorphism $$(M, R)\cong \bigoplus_{i=1}^l (\mathbb{Z}[\zeta]/(\pi_i^{n_i}), R_i)$$ translates to a quasigroup isomorphism of the form \eqref{Eq:ModDecompMTS}, and we may, without regard to potential (co)continuity issues associated with $F$, conflate direct sum in one category with direct product in another since our indices are finite.
\end{proof}

Because the functor of Theorem \ref{Th:CatRelat} is not an equivalence, there is not necessarily a one-to-one correspondence between subquasigroups of a linear MTS and its $\mathbb{Z}[\zeta]$-submodules.  The following corollary, therefore, should not be read as a complete characterization of the subquasigroup lattice of a linear MTS.

\begin{cor}\label{Cor:Subsystems}
Let $M$ be a finite abelian group, and $\emph{Lin}(M, R)$ a linear MTS with direct product decomposition 
$$\emph{Lin}(M, R)\cong \prod_{i=1}^l\emph{Lin}\left(\mathbb{Z}[\zeta]/(\pi_i^{n_i}), R_i\right).$$ Each factor $\emph{Lin}(\mathbb{Z}[\zeta]/(\pi_i^{n_i}), R_i)$ (and by the universal property of the direct product, $\emph{Lin}(M, R)$ itself) contains a chain of subquasigroups corresponding to the ideal lattice 
\begin{equation}\label{Eq:IdealChain}
\left(0+(\pi_i^{n_i})\right)\leq \left(\pi_i^{n_{i}-1}+(\pi_i^{n_i})\right)\leq\cdots\leq \left(\pi_i+(\pi_i^{n_i})\right)\leq \left(1+(\pi_i^{n_i})\right)    
\end{equation}
of the quotient ring $\mathbb{Z}[\zeta]/(\pi_i^{n_i})$.
\end{cor}

\begin{proof}
Let $(M, R)$ be a $\mathbb{Z}[\zeta]$-module with submodule $(M^\prime, R^\prime)$.  It suffices to show that $\text{Lin}(M^\prime, R^\prime)$ is a subquasigroup of $\text{Lin}(M, R)$.  This follows immediately from Theorem \ref{Th:CatRelat}, for the $\mathbb{Z}[\zeta]$-module inclusion $M^\prime\hookrightarrow M$ specifying $M^\prime$ as a submodule is an injective quasigroup homomorphism under $F.$    
\end{proof}

The group of units of the Eisenstein integers, $\mathbb{Z}[\zeta]^\times=\{\pm 1, \pm \zeta, \pm \overline{\zeta}\}$, consists of the sixth roots of unity in $\mathbb{C}^\times$.
Primes in $\mathbb{Z}[\zeta]$ belong to, up to association by units, three pairwise disjoint classes \cite[Prop.~9.1.4]{Ireland}:

\begin{enumerate}
\item $\pi,$ where  $\pi N \equiv 1\mod 3$ is prime in $\mathbb{Z}$;
\item $p$, where  $p$ is prime in $\mathbb{Z},$ and $p\equiv 2\mod 3$;
\item $1+\zeta$. 
\end{enumerate}  

\noindent Do note that none of the primes in class (1) are \emph{rational}, i.e., equal to $p+0\zeta$ for some $p\in\mathbb{Z}$.  The second class, however, does consist entirely of rational prime representatives.

\subsubsection{Quotients of $\mathbb{Z}[\zeta]$}
The classification of linear Mendelsohn triple systems depends upon the nature of the quotients presented by $\eqref{Eq:StructThmfgEisMods}$.  Our discussion is taken from \cite{Bucaj, Misa}, and our proofs are abridged adaptations these authors' results.  Note that in these works, the cubic root of unity $\omega=e^{2\pi i/3}$ is the preferred algebra generator.

\indent For an arbitrary commutative, unital ring $S$, we may define $$S[\zeta]=S[X]/(X^2-X+1).$$  One may consider this $S$-algebra as the subalgebra of $M_2(S)$ consisting of matrices of the form \eqref{Eq:EisIntMatRing}.  

\indent Before classifying quotients of $\mathbb{Z}[\zeta]$ by primary ideals, we collect some lemmas on more general Eisenstein quotients and the nature of irrational Eisenstein primes.

\begin{lmm}\label{Lmm:Eisgcd}
Let $a+b\zeta\in \mathbb{Z}[\zeta]$, so that $\gcd(a, b)=1$.  Fix $k=(a+b\zeta)N=a^2+ab+b^2$.  Then $\mathbb{Z}[\zeta]/(a+b\zeta)\cong \mathbb{Z}/_k$.
\end{lmm}

\begin{proof}
Because $\gcd(a, b)=1$, $b$ is a unit in $\mathbb{Z}/_{k}$.  The map $$\mathbb{Z}[\zeta]\to \mathbb{Z}/_k; x+y\zeta\mapsto (x+y)-ab^{-1}y \mod k$$ is a surjective ring homomorphism with kernel $(a+b\zeta)$ (cf. \cite[Th.~2.5]{Bucaj}).  
\end{proof}

\begin{lmm}\label{Lmm:ConjNotAssoc}
Let $\pi\in\mathbb{Z}[\zeta]$ be prime so that $\pi N\equiv 1\mod 3$ is a rational prime.
\begin{itemize}
    \item[$(\textrm a)$]
    $\pi$ and its conjugate are not associates.  That is, there is no unit $u\in \mathbb{Z}[\zeta]^\times$ so that $u\pi=\overline{\pi}.$
    \item[$(\textrm b)$] Let $n\geq 1,$ and $\pi^n=a+b\zeta$.  Then $\gcd(a, b)=1$.
\end{itemize}
\end{lmm}

\begin{proof}
$(\textrm a)$  Let $\pi=x+y\zeta$.  Then $\overline{\pi}=x+y\overline{\zeta}=(x+y)-y\zeta$.  If we assume $u(x+y\zeta)=(x+y)-y\zeta$ for each $u\in \mathbb{Z}[\zeta]^\times$, we arrive at three types of contradictions:
\begin{enumerate}
    \item $\pi$ is a rational prime; 
    \item $\pi N=3z^2\equiv 0\mod3$ for some $z\in \{x, y\}$;
    \item $\pi N=z^2$ is prime in $\mathbb{Z}$ for some $z\in\{x, y\}$.
\end{enumerate}
When $u=1$, we get a type-(1) contradiction.  Type-(2) arises for $u=-1, \zeta, \overline{\zeta}$.  Assuming $u=-\zeta, -\overline{\zeta}$ results in type-(3) contradictions.
\vskip3mm

$(\textrm b)$ To the contrary, assume $\gcd(a, b)>1$.  Then there is some prime integer $p \mid a, b$ so that $\pi^n=a+b\zeta=p(r+s\zeta)$ where $r, s\in \mathbb{Z}$.  If $p=3$, then $(\pi N)^n=(\pi^n)N=9(r^2+rs+s^2)\equiv 0 \mod 3$ is a contradiction.  Assuming $p\equiv 2\mod 3$ yields, because $p$ would itself be an Eisenstein prime, distinct prime factorizations in the UFD $\mathbb{Z}[\zeta]:$ $a+b\zeta=\pi^n=p\left(\prod \sigma_i \right)$, yet another contradiction.  We are left to consider the case $p\equiv 1\mod 3.$  It follows $p=\sigma\overline{\sigma}$ splits into irrational primes.  By unique factorization, $\pi$ and $\sigma$ must be associaties, that is, $\sigma=u\pi$ for some $u\in\mathbb{Z}[\zeta]^\times,$ and what's more, $p=\sigma\overline{\sigma}=(u\pi)(\overline{u\pi})=\pi\overline{\pi}.$  Invoking the fact that $\mathbb{Z}[\zeta]$ is a UFD once again, $\pi$ and $\overline{\pi}$ must be associate, contradicting $(\textrm a)$.         
\end{proof}

\begin{lmm}\label{Lmm:EisIntRatQuo}
Suppose $a+b\zeta\in \mathbb{Z}[\zeta]$ is associated to a positive, rational integer $k$.  Then $\mathbb{Z}[\zeta]/(a+b\zeta)=\mathbb{Z}[\zeta]/(k)\cong \mathbb{Z}/_{k}[\zeta].$ 
\end{lmm}

\begin{proof}
Note that $(k)=(a+b\zeta)$.  The map $$x+y\zeta\mapsto (x\mod k)+(y\mod k)\zeta$$ is a surjective homomorphism with kernel $(k)$ (cf. \cite[Th.~2.1]{Bucaj}).
\end{proof}

\begin{thm}\label{Th:EisQuos}
Let $n\geq 1$.  We may partition the collection of primary ideals in $\mathbb{Z}[\zeta]$ into four classes.  These, along with their respective quotient rings are summarized below.
\begin{enumerate}
    \item[$(\mathrm a)$] Suppose $\pi$ is an irrational prime.  That is, $p=\pi N\equiv 1 \mod 3$ is prime in $\mathbb{Z}$.  Then $\mathbb{Z}[\zeta]/(\pi^n)\cong \mathbb{Z}/_{p^n}$.   
    \item[$(\mathrm b)$] Suppose $p\equiv 2\mod 3$ is a rational prime in $\mathbb{Z}[\zeta]$.  Then $\mathbb{Z}[\zeta]/(p^n)\cong \mathbb{Z}/_{p^n}[\zeta]$
    \item[$(\mathrm c)$] Suppose $n=2k$ is even.  Then $\mathbb{Z}[\zeta]/((1+\zeta)^n)\cong \mathbb{Z}/_{3^k}[\zeta]$.
    \item[$(\mathrm d)$] Suppose $n=2k+1$ is odd.  Then $$\mathbb{Z}[\zeta]/((1+\zeta)^n)\cong\mathbb{Z}[X]/(3^{k+1}, 3^kX, X^2-X+1).$$ 
\end{enumerate}
\end{thm}

\begin{proof}
$(\textrm{a})$ This is a direct consequence of Lemmas \ref{Lmm:Eisgcd}, \ref{Lmm:ConjNotAssoc}.

\vskip 3mm

$(\textrm{b})$ This is a direct consequence of Lemma \ref{Lmm:EisIntRatQuo}.

\vskip 3mm

$(\textrm c)$ Since $(1+\zeta)^n=((1+\zeta)^2)^k=(3\zeta)^k=3^k\zeta^k$, $(1+\zeta)^n$ is associated with the positive, rational integer $3^k$.  By Lemma \ref{Lmm:EisIntRatQuo}, $\mathbb{Z}[\zeta]/(1+\zeta)^n\cong\mathbb{Z}/_{3^k}[\zeta]$. 
\vskip 3mm

$(\textrm{d})$  We denote our ideal $J=\left((1+\zeta)^{2k+1}\right)=\left(3^k\zeta^k(1+\zeta)\right)=\left(3^k(1+\zeta)\right)$.  It suffices to show that
\begin{equation}\label{Eq:OddRamifiedCosets}
C=\{(a+b\zeta)+J\mid 0\leq a\leq 3^{k+1}-1, 0\leq b\leq 3^k-1\}
\end{equation}
constitutes a set of distinct coset representatives, as $$\left|\mathbb{Z}[\zeta]/(1+\zeta)^{2k+1}\right|=(1+\zeta)^{2k+1}N=3^{2k+1}=|C|.$$
Suppose that $(a-c)+(b-d)\zeta\in J$, where $0\leq a, c\leq 3^{k+1}-1$, and $0\leq b, d\leq 3^k-1.$  Then 
\begin{equation}\label{Eq:OddPrCong}
(a-c)+(b-d)\zeta=z3^k(1+\zeta)
\end{equation}
for some $z\in \mathbb{Z}[\zeta]$.  Multiply both sides of \eqref{Eq:OddPrCong} by $(2-\zeta)$.  Then
\begin{align*}
    z3^{k+1}&=\left[2(a-c)+(b-d)\right]+\left[(b-d)-(a-c)\right]\zeta\\
    &=:X+Y\zeta.
\end{align*}
Thus, $3^{k+1}\mid X, Y$; so $3^{k+1}\mid X+2Y=3(b-d)\implies 3^k\mid (b-d)$.  By assumption on the range of values for $b$ and $d$, $b=d$.  Since $3^{k+1}\mid Y=(b-d)-(a-c)=0-(a-c)=c-a$, we may conclude $a=c$.  Hence, $\eqref{Eq:OddRamifiedCosets}$ does in fact consist of distinct coset representatives.
\end{proof}

\section{Distributive, Non-ramified MTS}\label{Sec:BMM}

\subsection{Direct product decomposition}
For the remainder of the paper, the reader may assume $f(X)$ refers to the polynomial $X^2-X+1$.  
In light of the fact that, as a rational prime, $3$ is ramified in $\mathbb{Z}[\zeta]$, we make the following definition.

\begin{definition}\label{Defn:DNRMTS}
Let $\text{Lin}(M, R)$ be a linear MTS of order $n$.  If $3\nmid n$, we say that $\text{Lin}(M, R)$ is a \emph{distributive, non-ramified} (DNR, for short) MTS. 
\end{definition}

The following lemma and its proof appear under Lemma 2.2 in \cite{Donovan}. 

\begin{lmm}[\cite{Donovan}]\label{Lmm:DonovanLemma}
Let $p$ be a prime, $n\geq 1.$ Then $f(X)\in \mathbb{Z}/_{p^n}[X]$ has
    \begin{enumerate}
        \item[$(\mathrm a)$] two distinct roots modulo $p^n$ if $p\equiv 1\mod 3;$
        \item[$(\mathrm b)$] no roots $p^n$ if $p\equiv 2\mod 3;$
        \item[$(\mathrm c)$] a double root modulo $3$ and no roots modulo $3^n$ for $n\geq 2.$
    \end{enumerate}
\end{lmm}

\begin{prop}\label{Prop:onemodthreeclassif}
Suppose $\pi\in \mathbb{Z}[\zeta]$ is an irrational prime with $p=\pi N\equiv 1\mod 3,$ and that $n\geq1$.  Let $\emph{Lin}(\mathbb{Z}[\zeta]/(\pi^n), R)$ be a linear Mendelsohn pique, $a^{\pm1}$ the distinct roots of $f(X)\in \mathbb{Z}/_{p^n}[X]$.  Then $\emph{Lin}(\mathbb{Z}[\zeta]/(\pi^n), R)$ is isomorphic to either $ \emph{Lin}(\mathbb{Z}/_{p^n}, a)$ or  $\emph{Lin}(\mathbb{Z}/_{p^n}, a^{-1})$, where $a^{\pm1}$ also stand for right multiplication by the corresponding roots.    
\end{prop}

\begin{proof}
It is clear from Theorem \ref{Th:EisQuos}(a) that $R$ must correspond to an abelian group automorphism of $\mathbb{Z}/_{p^n},$ all of which arise from multiplication by units.  Then $\text{Lin}(\mathbb{Z}[\zeta]/(\pi^n), R)$ is Mendelsohn if and only if the unit in question is a root of $f(X)$. 
\end{proof}

\indent Next, we generalize the argument of \cite[Th.~2.12]{Donovan} from matrices over finite fields to matrices over $\mathbb{Z}/_{p^n}$.  The presence of zero divisors means that extra care must be taken when addressing characteristic and minimal polynomials of matrices.

\begin{lmm}\label{Lmm:CharPoly}
Let $p$ be a prime congruent to $2$ modulo $3$, and $n\geq 1$.  Suppose $S$ is any of the following rings: $\mathbb{Z},$ $\mathbb{Z}/_{p^n}$, or $\mathbb{Z}/_{3^n}$.  For any $A\in M_2(S)$ such that $f(A)=0$, we have $1=\det(A)=\emph{Tr}(A)$. Therefore, $f(X)$ is equal to $\det(I_2X-A),$ the characteristic polynomial of $A$. 
\end{lmm}

\begin{proof}
We treat $S=\mathbb{Z}/_3$ as a special case.  Since $S$ is a field in this scenario, the Cayley-Hamilton theorem dictates that $f(X)$ must be the characteristic polynomial of $A$, for it is monic and its degree matches the dimension of $A$.  

\indent Now,
let $$A=\left(\begin{array}{cc}
   a  & b \\
   c  & d
\end{array}\right).$$  We will establish some facts that hold over any commutative, unital ground ring.  By Lemma \ref{Lmm:rootschar3}, $A$ is invertible and $\det(A)\in S^\times$.  Moreover, 
\begin{equation}
    (\det(A))^{-1}\left(\begin{array}{cc}
        d & -b \\
        -c & a
    \end{array}\right)=A^{-1}=I_2-A=\left(\begin{array}{cc}
      1-a   & -b \\
        -c & 1-d
    \end{array}\right).
\end{equation}
Then $b=(\det(A))^{-1}b,$ $c=(\det(A))^{-1}c,$ and $1-a=d(\det(A))^{-1}.$  If at least one of $b$ or $c$ is a unit in $S,$ then $\det(A)=1$.  If $S$ is an integral domain and at least one of $b$ or $c$ is nonzero, $\det(A)=1$.  If $\det(A)=1,$ then $1-a=d\implies \text{Tr}(A)=1$.  We proceed to a case-by-case argument by contradiction.

\indent Suppose $S=\mathbb{Z}$, a domain.  If both $b=c=0$, then $a$ and $d$ are integral roots of $f(X)$, but $f(X)$ is irreducible over $\mathbb{Z}$, a contradiction.

\indent Assume $S=\mathbb{Z}/_{p_n}$.  For $b$ and $c$ to both be non-invertible means that $p\mid b, c$.  Hence, if we take the entries of $A$ modulo $p$, we get a diagonal matrix annihilated by $f(X),$ contradicting Lemma \ref{Lmm:DonovanLemma}.($\mathrm{b}$).

\indent Let $S=\mathbb{Z}/_{3^n}$, where $n\geq 2.$  By analyzing congruence relations for $2\times2$ integral matrices modulo $9$ in GAP, we were able to verify that no matrix in $M_2(\mathbb{Z}/_9)$ annihilated by $f(X)$ has both off-diagonal entries divisible by $3$.  Conclude that the result holds for $S=\mathbb{Z}/_9$.  Suppose that $n\geq 3,$ and that $3\mid b, c$.  Then taking the entries of $A$ modulo $9$ yields a matrix annihilated by $f(X)$ over $\mathbb{Z}/_{9}$ with off-diagonal entries divisible by $3,$ contradicting our calculations in GAP.      
\end{proof}

\begin{prop}\label{Prop:twomodthreeclassif}
Suppose $p \equiv 2 \mod3$ is a rational prime in $\mathbb{Z}[\zeta]$ and $n\geq 1$.  Let $\emph{Lin}(\mathbb{Z}[\zeta]/(p^n), R)$ be a linear MTS, and $M=(\mathbb{Z}/_{p^n})^2.$  Then $\emph{Lin}(\mathbb{Z}[\zeta]/(p^n), R)\cong \emph{Lin}(M, T)$, where 

\begin{equation}\label{Eq:CompMat}
    T=\left(\begin{array}{cc}
        0 & -1 \\
        1 & 1
    \end{array}\right)
\end{equation}
    is the companion matrix of $f(X)\in\mathbb{Z}/_{p^n}[X]$.
\end{prop} 

\begin{proof}
That $\mathbb{Z}[\zeta]/(p^n)$ and $(\mathbb{Z}/_{p^n})^2$ are isomorphic as abelian groups is a direct consequence of Theorem \ref{Th:EisQuos}.(b).
Let $A\in M_2(\mathbb{Z}/_{p^n})$ be a matrix representation of $R$.  By Theorem \ref{Th:K&N}, it suffices to show $A$ is similar to \eqref{Eq:CompMat}.  By Lemma \ref{Lmm:CharPoly}, $f(X)$ is the characteristic polynomial of $A$.  We invoke Theorem $1$ in \cite{Prokip}, which states that $A$ is similar to the companion of its characteristic polynomial if and only if there is some $v\in M$ so that the matrix 
\begin{equation}\label{Eq:IMatrix}
\left(\begin{array}{c} v \\ vA\end{array}\right)   
\end{equation}
is invertible; this result applies to matrices with elements from arbitrary commutative, unital rings.  Recall $\mathbb{Z}/_{p^n}$ is a local ring with unique maximal ideal $\mathfrak{m}$ generated by $p+(p^n)$.  Taking the entries of $A$ modulo $p$ gives rise to an action on the quotient $M/M\mathfrak{m}\cong (\mathbb{Z}/_p)^2$, a vector space.  This automorphism, we'll call it $\overline{A},$ also has $f(X)$ as its characteristic polynomial, and because $f(X)$ does not split in $\mathbb{Z}/_p$, it has no eigenvalues.  Hence, $\{\overline{v}, \overline{v}\overline{A}\}$ is a basis for $M/M\mathfrak{m}$ whenever $\overline{v}\neq0$.  By Nakayama's lemma (applied to the special case of local rings), such a basis lifts to a minimal generating set $\{v, vA\}$ of $M$, and minimal generating sets for free modules over local rings are bases \cite[Th.~2.3]{Matsumura}.  Conclude that there is some $v\neq0$ for which the matrix \eqref{Eq:IMatrix} is invertible, and that $A$ is similar to \eqref{Eq:CompMat}.     
\end{proof}

\begin{remark}\label{Rmk:TwoMod3Rep}
Given a prime $p\equiv 2\mod 3$ and $n\geq 1$, we take the linear MTS $\text{Lin}((\mathbb{Z}/_{p^n})^2, T)$ (cf. Prop. \ref{Prop:twomodthreeclassif}) as the standard representative for the single isomorphism class of distributive Mendelsohn quasigroups on $(\mathbb{Z}/_{p^n})^2$.  We adopt the notation $\text{Lin}(\mathbb{Z}/_{p^n}[\zeta])$ when referring to this quasigroup. 
\end{remark}

\begin{thm}\label{Th:DirectSumDecomp}
Let $Q$ be a distributive Mendelsohn quasigroup of order\\ $n=p_1^{r_1}\cdots p_k^{r_k}q_1^{s_1}\cdots q_{l}^{s_l}$, where each $p_i$ is a prime congruent to $1$ modulo $3$, and each $q_i$ is a prime congruent to $2$ modulo $3$.  Then
\begin{equation}\label{Eq:DirectSumDecomp}
    Q\cong \prod_{i=1}^k\left(\prod_{j}\emph{Lin}\left(\mathbb{Z}/(p_i^{t_j}), a_j\right)\right)\times\prod_{i=1}^l\left(\prod_j\emph{Lin}\left(\mathbb{Z}/_{q_i^{u_j}}[\zeta]\right)\right), 
\end{equation}
where each $a_j$ is a root of $f(X)$ modulo $p_i^{t_j},$ $\sum_jt_j=r_i$, and $\sum_ju_j=s_i$.
\end{thm}

\begin{proof}
The direct product decomposition afforded by Fischer-Galkin-Smith, (Theorem \ref{Th:FSG}) contains no Moufang loop factors because $3\nmid n$.  Therefore, $Q$ is a direct product of linear MTS. Propositions \ref{Prop:onemodthreeclassif} and \ref{Prop:twomodthreeclassif} specify the possible quasigroup isomoprhism classes for the factors, and these agree with \eqref{Eq:DirectSumDecomp}.
\end{proof}

\subsection{Enumeration}
For a positive integer $n,$ let $d(m)$ denote the number of isomorphism classes of distributive MTS of order $m$. Let $l(m)$ denote the number of isomorphism classes of linear (distributive) MTS.  By the Chinese remainder and Fischer-Galkin-Smith Theorems, in order to compute arbitrary $d(m)$, it suffices to determine $l(p^n)=d(p^n)$ for all $p\neq 3, l(3^n)$, and $d(3^n)-l(3^n)$.  In \cite{Donovan}, the authors gave $d(p),$ and $d(p^2)$ for all primes \cite[Th.~2.12]{Donovan}.  Donovan et. al. also calculated $l(p^n)$ for $p^n<1000,$ and $d(3^3)$, $d(3^4)$.  Moreover, the realization of $d(3^5)$ is part of \cite{243}.  Theorem \ref{Th:DirectSumDecomp} allows us to calculate $d(p^n)$ for arbitrary powers, as long as $p\neq 3$.  Note that $d(p^{2k+1})=0$ whenever $k\geq 0,$ and $p\equiv 2\mod 3$ \cite[Lem.~2.9]{Donovan}.

We think of integer partitions as multisets $(X, \mu)$ of positive integers, where $X\subseteq \mathbb{Z}^+,$ and the function $\mu:X\to\mathbb{Z}^+$ determines the multiplicity of an integer in the partition.  Write $(X, \mu)\vdash n$ if $(X, \mu)$ is a partition of $n$.   Let $P(n)$ denote the number of integer partitions of $n$, and $P_E(n)$ signify the number of partitions consisting only of even parts.
\begin{thm}\label{Thm:EnumPrimeOrder}
Let $p\neq 3$ be prime, and $n\geq 1$.  
\begin{itemize}
    \item[$(\mathrm a)$] If $p\equiv 1\mod 3$, then 
    \begin{equation}\label{Eq:EnumPrimeOrder}
    d(p^n)=\sum_{(X, \mu)\vdash n}\left(\sum_{r\in X}\mu(r)+1\right).    
    \end{equation}
    \item[$(\mathrm b)$] If $p\equiv 2 \mod 3$, then $d(p^{n})=P_E(n)$.
\end{itemize}
\end{thm}

\begin{proof}
($\mathrm{a}$)  Let $Q$ be a distributive Mendelsohn quasigroup of order $p^n.$  By Theorem \ref{Th:DirectSumDecomp},
\begin{equation}\label{Eq:ArbEnumeratedQgp}
    Q\cong \prod_i\text{Lin}(\mathbb{Z}/_{p^{r_i}}, a_i),
\end{equation}
where $\sum_ir_i=n.$  Use $(X, \mu)$ to denote the partition of $n$ associated with this sum.  Fix $r\in X$.  Then $\mathbb{Z}/_{p^r}$ appears exactly $\mu(r)$ times as a factor in \eqref{Eq:ArbEnumeratedQgp}.  With two choices for MTS structure on $\mu(r)$ copies, there are 
\begin{align*}
    \frac{(\mu(r)+2-1)!}{\mu(r)!(2-1)!}&=\frac{(\mu(r)+1)!}{\mu(r)!}\\
    &=\mu(r)+1
\end{align*}
allowable isomorphism classes attached to  $(\mathbb{Z}/_{p^r})^{\mu(r)}$.  This count applies to each element of $X$, verifying the indexing on the innermost sum of \eqref{Eq:EnumPrimeOrder}, while the outermost indexing corresponds to the fact that we had $P(n)$ choices for the decomposition \eqref{Eq:ArbEnumeratedQgp}.   
\vskip 3mm
($\mathrm{b}$) Let $Q$ be a distributive Mendelsohn quasigroup of order $p^{2n}$.  We have 
\begin{equation}\label{Eq:ArbEnumeratedQgp2}
    Q\cong \prod_i\text{Lin}(\mathbb{Z}/_{p^{r_i}}[\zeta]),
\end{equation}
with $\left|\text{Lin}(\mathbb{Z}/_{p^{r_i}}[\zeta])\right|=p^{2r_i}$ for each $i$, and $\sum r_i=n$.  Then there is a one-to-one correspondence between choices of the decomposition \eqref{Eq:ArbEnumeratedQgp2} (i.e., isomorphism classes) and integer partitions of $2n$ consisting solely of even parts.     
\end{proof}

\section{The Ramified Case}\label{Sec:Ramified}

Classifying distributive MTS whose order is divisible by $3$ presents two issues: $1.)$ The fact that  $3=(1+\zeta)(1+\overline{\zeta})$ is ramified in $\mathbb{Z}[\zeta]$ complicates the direct product decomposition for CML-linear MTS, and $2.)$ the potential presence of nonassociative CML in the direct product decomposition voids our technique of relying solely on the representation theory of the Eisenstein integers.  Issue $2.)$ is beyond the scope of this paper.  This section elaborates issue $1.)$ and presents partial results towards a resolution.  By Corollary \ref{Cor:ModDecompMTS}, we may, at least, state the following decomposition theorem.

\begin{thm}\label{Th:Decomp3}
Let $Q$ be a linear Mendelsohn quasigroup of order $3^n$.  Then there is a finite sequence of positive integers $r_1, \dots, r_l$ such that 
\begin{equation}\label{Eq:Decomp3}
    Q\cong \prod_{i=1}^l\emph{Lin}(\mathbb{Z}[\zeta]/(1+\zeta)^{r_i}, R_i),
\end{equation}
where $\sum_i r_i=n$, and each $R_i$ is an automorphism of the abelian group underlying $\mathbb{Z}[\zeta]/(1+\zeta)^{r_i}$.
\end{thm}
What's more, we will prove that if $r_i=2k>2$ is even, then we have $\text{Lin}(\mathbb{Z}[\zeta]/(1+\zeta)^{r_i}, R_i)\cong \text{Lin}((\mathbb{Z}/_{3^k})^2, T),$ where $T$ is the companion matrix of $f(X)\in\mathbb{Z}/_{3^k}[X]$.  Proposition \ref{Prop:ThreeSquClassif}, equivalent to Theorem 2.12.(iii) in \cite{Donovan}, clarifies the case of $r_i=2.$  

\begin{prop}\label{Prop:ThreeSquClassif}
A Mendelsohn quasigroup $\emph{Lin}(\mathbb{Z}[\zeta]/(1+\zeta)^2, R)$ is isomorphic to either $\emph{Lin}((\mathbb{Z}/_3)^2, -I_2),$ or $\emph{Lin}((\mathbb{Z}/_3)^2, T),$ where 
\begin{equation}\label{Eq:CompMatMod3Squ}
    T=\left(\begin{array}{cc}
        0 & -1  \\
        1 & 1
    \end{array}\right)
\end{equation}
is the companion matrix of $f(X)$.
\end{prop}

\begin{proof}
By Theorem, \ref{Th:EisQuos}.(c), $(\mathbb{Z}/_3)^2$ is the abelian group underlying\\ $\mathbb{Z}[\zeta]/(1+\zeta)^2.$  Similarity classes of $(\mathbb{Z}/_3)^2$-automorphisms are represented by rational canonical matrices.  Matrices annihilated by $f(X)=(X+1)^2$ over $\mathbb{Z}/_3[X]$ have two possible RCFs.  Those are, $-I_2$ and $T$.
\end{proof}

The following lemma is needed for the case of classifying higher $k$-values.

\begin{lmm}\label{Lmm:NoDescToScalar} 
Let $k\geq 2$, $A\in M_2(\mathbb{Z}).$  If $A\equiv -I_2\mod 3,$ then $f(A)\not\equiv 0\mod 3^k$. 
\end{lmm}

\begin{proof}
The last paragraph of the proof of Lemma \ref{Lmm:CharPoly} furnishes a sufficient argument.
\end{proof}

\begin{prop}\label{Prop:EvenPower3Classif}
Let $k\geq 2$, and $\emph{Lin}(\mathbb{Z}[\zeta]/(1+\zeta)^{2k}, R)$ be a Mendelsohn quasigroup and $M=(\mathbb{Z}/_{3^k})^2$.  Then $\emph{Lin}(\mathbb{Z}[\zeta]/(1+\zeta)^{2k}, R)\cong \emph{Lin}(M, T)$, where 

\begin{equation}\label{Eq:CompMatMod3}
    T=\left(\begin{array}{cc}
        0 & -1 \\
        1 & 1
    \end{array}\right)
\end{equation}
is the companion matrix of $f(X)$. 
\end{prop}

\begin{proof}
Theorem \ref{Th:EisQuos}.(c) exhibits an abelian group isomorphism \\$\mathbb{Z}[\zeta]/(1+\zeta)^{2k}\cong (\mathbb{Z}/_{3^k})^2$.  Let $A$ denote a matrix representation of $R$.  Because $k\geq 2$, $f(X)$ has no roots in $\mathbb{Z}/_{3^k}$ (cf. Lemma \ref{Lmm:DonovanLemma}(c)), and $\det(I_2X-A)=f(X).$  Just as in the proof of Proposition \ref{Prop:twomodthreeclassif}, it suffices to show that $A$ is similar to the companion matrix of its characteristic polynomial by producing a $v\in(\mathbb{Z}/_{3^k})^2$ that makes the transpose of $(v\phantom{.} vA)$ invertible. 

\indent Define $\overline{A}$ to be the matrix consisting of the entries of $A$ taken modulo $3$.  Then $\overline{A}$ acts on $M/M\mathfrak{m}\cong (\mathbb{Z}/_3)^2$, where $\mathfrak{m}=(3+(3^k))$ is the unique maximal ideal of $\mathbb{Z}/_{3^k}$.  Recall that in $\mathbb{Z}/_3,$ $\det(XI_2-\overline{A})=f(X)=(X+1)^2.$  Let $W$ denote the eigenspace of $\overline{A}$ associated with $\overline{-1}$.  Then $\dim(W)=1$ or $2$.  By Lemma \ref{Lmm:NoDescToScalar}, $\overline{A}\neq-I_2$, so $\dim(W)=1$.  Pick $\overline{v}\in W^{\perp}$.  Then $\{\overline{v}, \overline{v}\overline{A}\}$ forms a basis of $M/M\mathfrak{m}$.  Once again, apply Nakayama's lemma, and we have $\{v, vA\}$ as a basis for the free $\mathbb{Z}/_{3^k}$-module $M$.  Conclude $A$ is similar to the companion of $f(X).$  
\end{proof}

\subsection{Odd powers of $1+\zeta$: mixed congruence}\label{Sec:OddPow}
In this section, we consider possible MTS structures on $\mathbb{Z}[\zeta]/(1+\zeta)^{2k+1}$. 
Upon examination of MTS on $\mathbb{Z}[\zeta]/(1+\zeta)^{2k+1}$ for arbitrary $k$, we are presented with a \emph{mixed congruence representation theory}.  To explain, let $G=\mathbb{Z}/_{3^k}\oplus \mathbb{Z}/_{3^{k+1}}$, $\pi_k:\mathbb{Z}\to \mathbb{Z}/_{3^{k}}$ and $\pi_{k+1}:\mathbb{Z}\to\mathbb{Z}/_{3^{k+1}}$ be natural projections, and $\pi=\pi_k\oplus \pi_{k+1}$.  Theorem \ref{Th:EisQuos}.(d) tells us it is enough to describe conjugacy classes of elements annihilated by $f(X)$ in $\text{Aut}(G)$.  We use the work presented in a concise, highly readable note by Hillar and Rhea to translate this problem into terms of matrix algebra \cite{Hillar}.  Consider the subset 
\begin{equation}\label{Eq:R3}
\mathcal{R}_3=\{(a_{ij})\in M_2(\mathbb{Z}): 3\mid a_{21}\}  
\end{equation}
of $2\times 2$ integral matrices.  This set is a subalgebra of $M_2(\mathbb{Z})$ \cite[Lm.~3.2]{Hillar}, and    
\begin{equation}\label{Eq:R3Homo}
\psi: \mathcal{R}_3\to \text{End}(G); A\mapsto ((x_1, x_2)\pi \mapsto ((x_1, x_2)A)\pi)
\end{equation}
is a surjective ring homomorphism \cite[Th.~3.3]{Hillar} with kernel 
\begin{equation}\label{Eq:R3Ker}
\{(a_{ij})\in M_2(\mathbb{Z}): 3^k\mid a_{11}, a_{12} \text{ and } 3^{k+1}\mid a_{21}, a_{22}\}   
\end{equation}
\cite[Lm.~3.4]{Hillar}.  Furthermore, an endomorphism $A^\psi\in \text{End}(G)$ is an automorphism if and only if $(A \mod 3)\in \text{GL}_2(\mathbb{Z}/_3)$ \cite[Th.~3.6]{Hillar}.
The enumeration techniques of Donovan et. al. classify all MTS on $\mathbb{Z}[\zeta]/(1+\zeta),$ $\mathbb{Z}[\zeta]/(1+\zeta)^{3},$ and $\mathbb{Z}[\zeta]/(1+\zeta)^{5}$.   

\begin{prop}\label{Prop:SmallOrderOddPowerClassif}
Let $\emph{Lin}(\mathbb{Z}[\zeta]/(1+\zeta)^{2k+1}, R)$ be a Mendelsohn quasigroup, $M_1=\mathbb{Z}/_3\oplus \mathbb{Z}/_9$, $M_2=\mathbb{Z}/_9\oplus\mathbb{Z}/_{27}$, $$T=\left(\begin{array}{cc}
    2 & -1 \\
    3 & -1
\end{array}\right),$$ and $\psi$ the ring homomorphism \eqref{Eq:R3Homo}.   
\begin{itemize}
    \item[$(\mathrm{a})$] If $k=0,$ then $\emph{Lin}(\mathbb{Z}[\zeta]/(1+\zeta)^{2k+1}, R)\cong\emph{Lin}(\mathbb{Z}/_3, -1).$
    \item[$(\mathrm b)$] If $k=1,$ Then $\emph{Lin}(\mathbb{Z}[\zeta]/(1+\zeta)^{2k+1}, R)\cong \emph{Lin}(M_1, T^\psi)$. 
    \item[$(\mathrm c)$] If $k=2,$ Then $\emph{Lin}(\mathbb{Z}[\zeta]/(1+\zeta)^{2k+1}, R)\cong \emph{Lin}(M_2, T^\psi)$.
\end{itemize}
\end{prop}

\begin{proof}
$(\mathrm{a})$ Quite simply, this follows from the field isomorphism\\ $\mathbb{Z}[\zeta]/(1+\zeta)\cong \mathsf{GF}(3)$

\vskip 3mm
$(\mathrm{b})$ Table $1$ in \cite{Donovan} tells us that there are $3$ isomorphism classes of linear MTS of order $3^3$.  The possible Eisenstein module decompositions are 
\begin{enumerate}
    \item $\left(\mathbb{Z}[\zeta]/(1+\zeta)\right)^3;$
    \item $\mathbb{Z}[\zeta]/(1+\zeta)^2\oplus \mathbb{Z}[\zeta]/(1+\zeta);$
    \item $\mathbb{Z}[\zeta]/(1+\zeta)^3;$
\end{enumerate}
The first must correspond to $-I_3\in M_3(\mathbb{Z}/_3)$, and the second, by Proposition \ref{Prop:EvenPower3Classif}, to $T\oplus (-1),$ where $T$ is the companion of $f(X)\in\mathbb{Z}/_3[X]$.  This leaves us with $\text{Lin}(M_1, T^\psi)$ as the only possible linear MTS structure on $(3)$.

\vskip 3mm
$(\mathrm{c})$ Donovan et. al.'s enumeration relates $7$ isomorphism classes of order-$3^5$ linear MTS.  There are $7$ integer partitions of $5$, so for each decomposition of the form \eqref{Eq:Decomp3}, there is exactly one isomorphism class.  In particular, the $\mathbb{Z}[\zeta]$-module $\mathbb{Z}[\zeta]/(1+\zeta)^5$ gives rise to one isomorphism class.
\end{proof}  

Based upon the evidence proffered by small examples in Proposition \ref{Prop:SmallOrderOddPowerClassif}, we make the following conjecture.

\begin{conj}\label{Conj:ClassifOddPow3}
Given $k\geq 3$ and an MTS $\emph{Lin}(\mathbb{Z}[\zeta]/(1+\zeta)^{2k+1}, R),$ there is a quasigroup isomorphism $\emph{Lin}(\mathbb{Z}[\zeta]/(1+\zeta)^{2k+1}, R)\cong \emph{Lin}(M, T^\psi),$ where $M=\mathbb{Z}/_{3^k}\oplus\mathbb{Z}/_{3^{k+1}},$ and $T^\psi$ is defined along the lines of Proposition \ref{Prop:SmallOrderOddPowerClassif}.  Moreover, the enumeration of order-$3^n$, linear MTS is given by $l(3^n)=P(n)$.
\end{conj}

\subsubsection{Lifting}
\indent The following discussion presents techniques for lifting the mixed congruence representation theory into $M_2(\mathbb{Z})$.  Take $n\geq 1$.  The projection $\mathbb{Z}\to\mathbb{Z}/_n$ extends naturally to a map between matrix groups $\text{SL}_2(\mathbb{Z})\to \text{SL}_2(\mathbb{Z}/_n)$, furnishing a short exact sequence 
\begin{equation}\label{Eq:PrinCongSubgp}
    1\to \Gamma(n)\to \text{SL}_2(\mathbb{Z})\to \text{SL}_2(\mathbb{Z}/_n)\to 1.
\end{equation}
The kernel $\Gamma(n)$ is referred to as the \emph{principal congruence subgroup of level} $n$.  A subgroup of $\text{SL}_2(\mathbb{Z})$ is a \emph{congruence subgroup} if it contains a principal congruence subgroup of any level.  The \emph{Hecke congruence subgroup of level} $n$, denoted $\Gamma_0(n)$, is the preimage of upper-triangular matrices under the projection $\text{SL}_2(\mathbb{Z})\to \text{SL}_2(\mathbb{Z}/_{n})$ \cite[Sec~1.2]{FCMF}.  Notice that the group of determinant-one units of the matrix subalgebra $\mathcal{R}_3$ is precisely $\Gamma_0(3)$.

\begin{lmm}\label{Lmm:ConjClassesSL2Z}
All elements in $M_2(\mathbb{Z})$ annihilated by the polynomial $f(X)$ have determinant $1$, and they lie in the same $\emph{SL}_2(\mathbb{Z})$-conjugacy class.
\end{lmm}

\begin{proof}
Suppose $A\in M_2(\mathbb{Z})$, and that $f(A)=0$. By Lemma \ref{Lmm:CharPoly}, we have $\det(A)=\text{Tr}(A)=1$.  Thus, $A\in \text{SL}_2(\mathbb{Z})$.  In \cite[Th.~1]{Chowla}, it is shown that the number of $\text{SL}_2(\mathbb{Z})$-conjugacy classes whose constituents have trace $t$ is equal to the number of equivalence classes of binary quadratic forms with discriminant $t^2-4$.  Up to equivalence, there is only one binary quadratic form with discriminant $1^2-4=-3.$  
\end{proof}

\begin{lmm}\label{Lmm:LowerTriMat}
Suppose $A\in \Gamma_0(3)$, and $f(A)=0$.  Then there exists some $P\in \Gamma_0(3)$ so that \begin{equation}\label{Eq:Lifting}
    P^{-1}AP=\left(\begin{array}{cc}
        2 & -1 \\
        3 & -1
    \end{array}\right).
\end{equation}
\end{lmm}

\begin{proof}
Lemma \ref{Lmm:ConjClassesSL2Z} tells of a $P\in \text{SL}_2(\mathbb{Z})$ for which \eqref{Eq:Lifting} holds.  We prove that any such $P$ must belong to $\Gamma_0(3)$.  Our argument relies on the fact that when the entries of $A$ and $P$ are taken modulo $3$, we get invertible matrices in $\text{SL}_2(\mathbb{Z}/_3)$.  Now,
\begin{equation*}
    P^{-1}AP\equiv \left(\begin{array}{cc}
        -1 & -1 \\
        0 & -1
    \end{array}\right) \mod 3.
\end{equation*}
We claim that either
\begin{equation}\label{Eq:UpperTriangMat1}
    A\equiv \left(\begin{array}{cc}
        -1 & -1 \\
       0  & -1
    \end{array}\right) \mod 3,
\end{equation}
or 
\begin{equation}\label{Eq:UpperTriangMat2}
    A\equiv \left(\begin{array}{cc}
        -1 & 1 \\
        0 & -1
    \end{array}\right) \mod 3.
\end{equation}
Notice that $f(A)=0\equiv 0 \mod 3.$  The only upper-triangular matrices in $\text{SL}_2(\mathbb{Z}/_3)$ annihilated by $f(X)$ are \eqref{Eq:UpperTriangMat1}, \eqref{Eq:UpperTriangMat2}, and $-I_2$.  However, since $f(A)=0\equiv 0\mod 9$ we cannot have $A\equiv -I_2\mod 3$, for this would contradict Lemma \ref{Lmm:NoDescToScalar}.  Thus, \eqref{Eq:UpperTriangMat1} or \eqref{Eq:UpperTriangMat2} holds.  A straightforward computation in GAP verifies that the only elements of $\text{SL}_2(\mathbb{Z}/_3)$ conjugating \eqref{Eq:UpperTriangMat1} to itself, or \eqref{Eq:UpperTriangMat2} to \eqref{Eq:UpperTriangMat1} contain a zero in the lower left-hand entry.  Thus, $P$ is upper-triangular in $\text{SL}_2(\mathbb{Z}/_3)$, so $P\in \Gamma_0(3).$   
\end{proof}

For our next proposition, we define 
\begin{equation}\label{Eq:LiftingClass}
\mathcal{C}=\{\alpha\in \text{Aut}(G)\mid \exists A\in \Gamma_0(3) \phantom{.} (A^\psi=\alpha \wedge A^2-A+I_2=0)\}.
\end{equation}

\begin{prop}\label{Prop:Lifting}
All members of $\mathcal{C}$ are conjugate.
\end{prop}

\begin{proof}
Let $\alpha\in\mathcal{C}$ and $A\in\Gamma_0(3)$ so that $A^\psi=\alpha,$ and $f(A)=0$ over $\mathbb{Z}$.  We show $\alpha$ is conjugate to
$$\gamma=\left(\begin{array}{cc}
    2 & -1 \\
     3 & -1
\end{array}\right)^\psi.$$
By Lemma \ref{Lmm:LowerTriMat}, there is some $P\in \Gamma_0(3)$ so that \eqref{Eq:Lifting} holds.  Therefore, $\gamma=(P^{-1}AP)^\psi=(P^\psi)^{-1}\alpha P^\psi$. 
\end{proof}

\begin{example}\label{Ex:MixedCongLift}
The key to proving Conjecture \ref{Conj:ClassifOddPow3} may be showing that all members of $\text{Aut}(G)$ annihilated by $f(X)$ belong to $\mathcal{C}$.  This is the case for $G=\mathbb{Z}/_3\oplus\mathbb{Z}/_9$.  Matrix representatives for the six automorphisms annihilated by $f(X)$ and corresponding lifts into $\Gamma_0(3)$ are given below:
\begin{align*}
    \left(\begin{array}{cc}
        2 & 2 \\
        3 & 8
    \end{array}\right)&\equiv \left(\begin{array}{cc}
        2 & -1 \\
        3 & -1
    \end{array}\right); \\
    \left(\begin{array}{cc}
       2  & 1 \\
        6 & 8
    \end{array}\right)&\equiv\left(\begin{array}{cc}
        2 & 1 \\
        -3 & -1
    \end{array}\right);\\
    \left(\begin{array}{cc}
        2 & 2 \\
        3 & 5
    \end{array}\right)&\equiv\left(\begin{array}{cc}
        5 & -1 \\
        21 & -4
    \end{array}\right);\\
     \left(\begin{array}{cc}
        2 & 1 \\
        6 & 5
    \end{array}\right)&\equiv\left(\begin{array}{cc}
        5 & 1 \\
        -21 & -4
    \end{array}\right);\\
     \left(\begin{array}{cc}
        2 & 2 \\
        3 & 2
    \end{array}\right)&\equiv\left(\begin{array}{cc}
        26 & -31 \\
        21 & -25
    \end{array}\right);\\
     \left(\begin{array}{cc}
        2 & 1 \\
        6 & 2
    \end{array}\right)&\equiv\left(\begin{array}{cc}
        26 & 31 \\
        -21 & -25
    \end{array}\right).
    \end{align*}
These matrix representations were obtained by implementing the procedure outlined in the proof of \cite[Th.~4.1]{Hillar} in GAP.  
\end{example}

\section{Combinatorial Properties of Linear MTS}\label{Sec:IncProp}

\subsection{Purity}\label{Sec:Purity}
\begin{definition}\label{Def:PureMTS}
A Mendelsohn triple system $(Q, \mathcal{B})$ is \emph{pure} if whenever $a, b\in Q$, we have $(a\phantom{.} b\phantom{.}ab)\in \mathcal{B}\implies (b\phantom{.} a\phantom{.}ab)\notin\mathcal{B}$.  
\end{definition}
Definition \ref{Def:PureMTS} is equivalent to anticommutativity of the underlying quasigroup.  The class of anticommutative quasigroups is specified in $\mathbf{Q}$ by the \emph{quasi-identity} (one defined by implication; cf. \cite[Def.~V.2.24]{Burris}) 
\begin{equation}\label{Eq:Quasi-Identity}
   xy=yx\implies x=y. 
\end{equation}
Hence, the class of anticommutative quasigroups forms a \emph{quasivariety}, or a class of algebras defined by quasi-identities.  Because Steiner triple systems are equivalent to MTS whose underlying quasigroup is commutative, purity can be seen as a quality of MTS which are maximally non-Steiner.

\indent Constructing linear MTS over finite fields, Mendelsohn proved the existence of pure MTS of order $p$ for $p\equiv 1\mod 3$ and of order $p^2$ for $p\equiv 2\mod 3$ \cite[Ths.~6,~7]{MTS}.  The existence spectrum of pure MTS was expanded using linear techniques \cite{Aczel}, until Bennett and Mendelsohn showed that the existence spectrum of pure MTS matches (with the exception of order $3$) that of all MTS.  That it to say, a pure MTS of order $n$ exists, with the exception of $n=1, 3, 6,$ if and only if $n\equiv 0, 1\mod 3$ \cite{PureMTS}. 

\begin{lmm}\label{Lmm:PureDirProd}
Let $Q_1$ and $Q_2$ be nonempty quasigroups.  The direct product $Q_1\times Q_2$ is anticommutative if and only if each of $Q_1$ and $Q_2$ are anticommutative.
\end{lmm}

\begin{proof}
Quasivarities are closed under direct product \cite[Th.~V.2.25]{Burris}.  Hence, the product of anticommutative quasigroups is anticommutative.  To obtain the other direction, we prove the contrapositive.  Suppose $Q_1$ is not anticommutative and that $a, b\in Q_1$ are distinct elements that commute.  Then for any $x\in Q_2$ (recall we assume each quasigroup to be nonempty), $(a, x)\neq (b, x)$, while  $(a, x)(b, x)=(ab, x^2)=(ba, x^2)=(b, x)(a, x).$
\end{proof}

\begin{prop}\label{Prop:PureMTS}
Every DNR MTS is pure.
\end{prop}

\begin{proof}
By Lemma \ref{Lmm:PureDirProd}, it suffices to show that for $p\equiv 1\mod 3$, $\text{Lin}(\mathbb{Z}/_{p^n}, a)$ is pure, and that for $p\equiv 2\mod 3$, $\text{Lin}(\mathbb{Z}/_{p^n}[\zeta])$ is pure.  In either case, $xy=yx$ means $xR+y(1-R)=yR+x(1-R)\implies x(2R-1)=y(2R-1)$, so it will suffice to show $2R-1$ is invertible.  

\indent The case of $(\mathbb{Z}/_p, a)$ is covered in \cite[Th.~6]{MTS}.  If $2a-1$ were not invertible in $\mathbb{Z}/_p$, then $2a-1=0\implies a=2^{-1}$.  By Lemma \ref{Lmm:rootschar3}, $f(2)=3=0$, a contradiction.  Furthermore, by way of contradiction, suppose $a\in \mathbb{Z},$ $f(a)\equiv 0\mod p^n,$ and $2a-1$ is not a unit modulo $p^n,$ where $n>1$.  This means that $2a-1=mp^k$ for some $m\in \mathbb{Z}$ and $1\leq k\leq n$, whereby $f(a)\equiv 0\mod p,$ and $2a-1\equiv0\mod p$, which contradicts the $n=1$ case. 

\indent Fix $p\equiv 2 \mod 3.$  We use $$A=\left(\begin{array}{cc}
    0 & -1 \\
    1 & 1
\end{array}\right)$$ to represent $\text{Lin}(\mathbb{Z}/_{p^n}[\zeta])$.  Then 
\begin{align*}
  \det\left(2A-I_2\right)&=\det\left(\begin{array}{cc}
     -1  & -2 \\
      2 & 1
  \end{array}\right)\\
  &=3,
\end{align*}
which is a unit modulo $p^n,$ so $2A-I_2$ is invertible.
\end{proof}

\begin{prop}\label{Prop:ImpureMTS}
Let $M$ be an abelian group of order $3^n$, $R\in \emph{Aut}(M)$ such that $\emph{Lin}(M, R)$ is Mendelsohn.  Then $M$ is not a pure MTS.
\end{prop}
\begin{proof}
By Lemma \ref{Lmm:PureDirProd}, it suffices to show that any MTS of the form \\$\text{Lin}(\mathbb{Z}[\zeta]/(1+\zeta)^n, R)$ is not pure.  Note $\mathbb{Z}[\zeta]/(1+\zeta)\cong \mathbb{Z}/_3,$ so if $n=1,$ the system is Steiner.     

\indent Henceforth, assume $n\geq 2$.  According to the discussion of mixed congruence in Section \ref{Sec:OddPow}, whether $n$ is even or odd, $R$ can be represented by a $2\times 2$ integral matrix; call it $A$.  By Lemma \ref{Lmm:CharPoly} ---dropping the second row of $A$ down a level of congruence if $n$ is odd--- there is some $k\geq1$ such that $\det(A)\equiv\text{Tr}(A)\equiv 1\mod 3^k$.  Consider $$P:=2A-I_2=\left(\begin{array}{cc}
    2a_{11}-1 & 2a_{12}  \\
    2a_{21} & 2a_{22}-1
\end{array}\right).$$       

\noindent We claim that $P$ is not invertible.  It suffices to show $\det(P)\equiv 0\mod 3$.  In fact,
\begin{align*}
    \det(P)&=4a_{11}a_{22}-2a_{11}-2a_{22}+1-4a_{12}a_{21}\\
    &=4(a_{11}a_{22}-a_{12}a_{21})-2(a_{11}+a_{22})+1\\
    &= 4\det(A)-2\text{Tr}(A)+1\\
    &\equiv 3 \mod 3^k\\
    &\equiv 0 \mod 3.
\end{align*}
With $P$ singular, we may choose some nonzero $$v\in \text{Ker} P=\{v\in M\mid vA+vA=v\}.$$  Since $2\nmid |M|,$ $v\neq -v$.  Moreover, $v(-v)=vA-v+vA=v-v=0=-v+v=(-v)A+v-vA=(-v)v$.  Conclude there exist distinct vectors commuting under the quasigroup operation given by $R$. 
\end{proof}

\begin{thm}\label{Th:PureMTSCharac}
Let $M$ be a finite abelian group of order $n$, and $\emph{Lin}(M, R)$ a linear MTS.  Then $\emph{Lin}(M, R)$ is pure if and only if $3\nmid n$.  That is, a linear MTS is pure if and only if it is non-ramified.   
\end{thm}

\begin{proof}
Let $\text{Lin}(M, R)=\prod_i Q_i$ be the decomposition of Corollary \ref{Cor:ModDecompMTS}.  By Lemma \ref{Lmm:PureDirProd}, $\text{Lin}(M, R)$ is pure if and only if each $Q_i$ is anticommutative.  Propositions \ref{Prop:PureMTS} and \ref{Prop:ImpureMTS} reveal that a given factor $Q_i$ is anticommutative if and only if $3\nmid |Q_i|.$  Therefore, $\text{Lin}(M, R)$ is pure if and only if each $Q_i$ is anticommutative if and only if $3\nmid |Q_i|$ for all $i$ if and only if $3\nmid n$. 
\end{proof}

\newpage
\subsection{The Converse of Linear MTS}\label{Sec:SelfConvMTS}

\subsubsection{Self-Converse MTS}
Let $*:Q^2\to Q$ be a binary operation on a set $Q,$ and $\tau:(x, y)\mapsto(y, x)$ the \emph{twist} on $Q^2$.  We define the \emph{opposite} of $*$ to be the binary operation $\tau *:(x, y)\mapsto y*x$.  Given a quasigroup $(Q, \cdot, /, \backslash),$ its \emph{opposite quasigroup} is the quadruple $(Q, \circ, \backslash\!\backslash, /\!/)$, where $\circ,$ $\backslash\!\backslash$, and $/\!/$ denote the opposite operations of multiplication, left division, and right division respectively.  Recall that a Mendelsohn quasigroup has the form $(Q, \cdot, \circ, \circ)$, so its opposite will be $(Q, \circ, \cdot, \cdot)$.  This has a combinatorial interpretation.   

\begin{definition}\label{Def:CMTS}
Let $(Q, \mathcal{B})$ be an MTS.  The \emph{converse} of $(Q, \mathcal{B})$ is given by $(Q, \mathcal{B}^{-1})$, where $\mathcal{B}^{-1}=\{(b\phantom{.}a\phantom{.} ab)\mid (a\phantom{.} b\phantom{.} ab)\in\mathcal{B}\}$.   
\end{definition}

\begin{definition}\label{Def:S-CMTS}
Let $(Q, \mathcal{B})$ be an MTS with corresponding quasigroup $(Q, \cdot, \circ, \circ)$.  We say that $(Q, \mathcal{B})$ is \emph{self-converse} if  $(Q, \cdot, \circ, \circ)$ is isomorphic to its opposite $(Q, \circ, \cdot, \cdot)$.  A CML-linear MTS $\text{Lin}(M, R)$ is self-converse if and only if $\text{Lin}(M, R)\cong\text{Lin}(M, 1-R)=\text{Lin}(M, R^{-1})$.
\end{definition}

The term self-converse comes from graph theory.  The converse of a digraph consists of the same set of points, with arrows reversed.  In fact, in order to obtain the converse of an MTS, one may take the Cayley graph with respect to multiplication in the underlying quasigroup and reverse arrows.  The examination of self-converse MTS first appears in the literature with Colbourn and Rosa's survey on directed triple systems \cite{Colbourn}.  Just as in the case of pure systems, it was shown that the existence spectrum of self-converse MTS matches the existence spectrum of arbitrary MTS \cite{SCMTS}.  We characterize self-converse DNR MTS and conjecture a complete characterization of entropic self-converse MTS. 

\begin{lmm}\label{Lmm:DirProdSelcConv}
Let $\emph{Lin}(M_1, R_1)$ and $\emph{Lin}(M_2, R_2)$ be Mendelsohn quasigroups. The direct product $\prod_i\emph{Lin}(M_i, R_i)$ is self-converse if and only if each factor $\emph{Lin}(M_i, R_i)$ is self-converse.
\end{lmm}

\begin{proof}
Begin by noticing $\prod_i\text{Lin}(M_i, R^{\pm 1}_i)\cong \text{Lin}(M_1\times M_2, R_1^{\pm 1}\times R^{\pm 1}_2)$.  We can extend isomorphisms $f_i:\text{Lin}(M_i, R_i)\to \text{Lin}(M_i, R_i^{-1})$ uniquely to the product $f_1\times f_2:\prod_i\text{Lin}(M_i, R_i)\to\prod_i\text{Lin}(M_i, R^{-1}_i)$; this proves sufficiency.  Conversely, we can restrict an isomorphism $f:\prod_i\text{Lin}(M_i, R_i)\to\prod_i\text{Lin}(M_i, R_i^{-1})$ to isomorphisms $f\!\mid_{0\times M_i}:\text{Lin}(M_i, R_i)\to\text{Lin}(M_i, R^{-1}_i),$ verifying necessity.
\end{proof}

\begin{thm}\label{Th:Self-ConverseCharac}
Let $M$ be a finite abelian group of order $n$, and $\emph{Lin}(M, R)$ a DNR MTS.  Then $\emph{Lin}(M, R)$ is self-converse if and only if for each prime $p\mid n,$ we have $p\equiv 2\mod 3.$ 
\end{thm}

\begin{proof}
The MTS $\text{Lin}(M, R)$ is self-converse if and only if each factor in its direct product decomposition \eqref{Eq:DirectSumDecomp} is self-converse.  Propositions \ref{Prop:onemodthreeclassif} and \ref{Prop:twomodthreeclassif} reveal that this is possible if and only if each factor is of the form $\text{Lin}(\mathbb{Z}/_{p^n}[\zeta])$ for some $p\equiv 2\mod 3.$
\end{proof}

\indent Conjecture \ref{Conj:ClassifOddPow3} implies the following:

\begin{conj}\label{Self-Conv}
Let $M$ be a finite abelian group, and $\emph{Lin}(M, R)$ an MTS.  Then $\emph{Lin}(M, R)$ is self-converse if and only if no prime dividing the order of $M$ is congruent to $1$ modulo $3$.
\end{conj}

\subsubsection{Self-Orthogonal MTS}

\begin{definition}\label{Def:OrthLSqu}
Two quasigroups on the same set of symbols $(Q, \cdot)$ and $(Q, *)$ are \emph{orthogonal} if for every $a, b\in Q$ the system of equations 
\begin{align*}
    x\cdot y&=a\\
    x*y&=b
\end{align*}
as a unique solution.  A quasigroup $Q$ is \emph{self-orthogonal} if it is orthogonal to its opposite.
\end{definition}

\begin{lmm}\label{Lmm:NSelfOrth}
Let $(Q, \cdot, /, \backslash)$ be a quasigroup with distinct, commuting elements.  Then $(Q, \cdot, /, \backslash)$ is not self-orthogonal.
\end{lmm}

\begin{proof}
Choose distinct, commuting elements $q, q^\prime\in Q,$ and fix $a=q\cdot q^\prime=q^\prime \cdot q.$  With respect to the system of equations
\begin{align*}
    x\cdot y&=a\\
    x\circ y&=a,
\end{align*}
the ordered pairs $(q, q^\prime), (q^\prime, q)$ constitute distinct solutions.
\end{proof}

The above lemma demonstrates that self-orthogonal quasigroups are necessarily anticommutative.  In general, the converse does not hold.  However, Di Paola and Nemeth showed that any MTS linear over a field which is pure (anticommutative) is necessarily self-orthogonal \cite[Th.~7]{Aczel}.  We prove that their result extends to quotients of $\mathbb{Z}[\zeta]$ by primary ideals.

\begin{thm}\label{Th:SelfOrthCharac}
Let $M$ be a finite abelian group, and $\emph{Lin}(M ,R)$ a linear MTS.  Then $\emph{Lin}(M, R)$ is self-orthogonal if and only if it is non-ramified if and only if it is pure.
\end{thm}

\begin{proof}
\indent Note that a quasigroup is self orthogonal if and only if $x_1y_1=x_2y_2$ and $y_1x_1=y_2x_2 \implies x_1=x_2\text{ and }y_1=y_2$.  With this characterization, one may prove ---employing a quasivariety argument identical to that of the proof of Lemma \ref{Lmm:PureDirProd}--- that the direct product of quasigroups is self-orthogonal if and only if each factor is self-orthogonal.  Thus, the fact that ramified linear MTS are not self-orthogonal follows directly from Proposition \ref{Prop:ImpureMTS} and Lemma \ref{Lmm:NSelfOrth}.  This proves that all self-orthogonal linear MTS are non-ramified.

\indent Next, to prove the converse, we have to establish self-orthogonality of $\text{Lin}(\mathbb{Z}/_{p^n}, a)$ whenever $p\equiv 1\mod 3$, and $\text{Lin}(\mathbb{Z}/_{p^n}[\zeta])$ whenever $p\equiv 2\mod 3$.  But first, we make a few observations regarding an arbitrary linear MTS $\text{Lin}(M, R)$.  Suppose $x_1, x_2, y_1, y_2\in M$ so that $x_1y_1=x_2y_2$ and $y_1x_1=y_2x_2$.  Recall $1-R=R^{-1}=-R^2$.  Then by hypothesis, 
\begin{equation}\label{Eq:SelfOrth1}
  x_1R-y_1R^2=x_2R-y_2R^2,
\end{equation}
and
\begin{equation}
  y_1R-x_1R^2=y_2R-x_2R^2\label{Eq:SelfOrth2}.
\end{equation}
Now multiply \eqref{Eq:SelfOrth1} and \eqref{Eq:SelfOrth2} by $R^{-1}$ and rearrange terms in order to obtain
\begin{equation}\label{Eq:SelfOrth3}
  x_1-x_2=(y_1-y_2)R,
\end{equation}
and
\begin{equation}
  y_1-y_2=(x_1-x_2)R\label{Eq:SelfOrth4}.
\end{equation}
Substituting the right-hand side of \eqref{Eq:SelfOrth4} in for $y_1-y_2$ in \eqref{Eq:SelfOrth3} yields
\begin{equation}\label{Eq:SelfOrth5}
x_1-x_2=(x_1-x_2)R^2.   
\end{equation}
Similarly,
\begin{equation}\label{Eq:SelfOrth6}
y_1-y_2=(y_1-y_2)R^2.
\end{equation}

\indent Now, assume we are working in $\text{Lin}(\mathbb{Z}/_{p^n}, a)$.  If $x_1-x_2\neq0,$ in order for \eqref{Eq:SelfOrth5} to hold, $a^2-1$ must be a zero divisor in $\mathbb{Z}/_{p^n}$, so $p\mid a^2-1$.  Thus, $a$ is a square root of $1$ modulo $p$, and if this is the case, then $a\equiv \pm 1 \mod p$.  We also have $a^2-a+1\equiv0\mod p$.  Clearly then, $a$ cannot be congruent to $1$, so it must be congruent to $-1$ modulo $p$.  However, $f(-1)\equiv 0\mod p$ if and only if $p=3$ by Lemma \ref{Lmm:rootschar3}.(b), contradicting our assumption that $p\equiv 1\mod 3$.  We are left to conclude $x_1=x_2$; similarly $y_1=y_2$.

\indent  The case of $\text{Lin}(\mathbb{Z}/_{p^n}[\zeta])$ is more concrete.  By Proposition \ref{Prop:twomodthreeclassif}, we can choose $R$ so that $R^2$ is represented by $$\left(\begin{array}{cc}
   -1  & -1 \\
    1 & 0
\end{array}\right).$$  Suppose $x_1-x_2=(a, b)\neq 0.$  That is, 
\begin{align*}
    (a, b)&=(a, b)R^2\\
    &=(b-a, -a).
\end{align*}
Then $b=-a=2a$, so that $0=3a$.  Since, $p$ does not divide $3,$ it is invertible in $\mathbb{Z}/_{p^n}$.  Conclude $a=-b=0$, a contradiction.  Thus, $x_1-x_2=0=y_1-y_2$.

\indent We have thus proven the equivalence of linear MTS being non-ramfied and self-orthogonal.  The equivalence of the former with purity was established in Theorem \ref{Th:PureMTSCharac}.
\end{proof}

\section{Reverse-Linearization}\label{Sec:RL}
We return to the alternative presentation $\mathbb{Z}[X]/(X^2+X+1)$ for the Eisenstein integers.  Since $X^2-X+1$ arises when constructing CML-linear quasigroups that satisfy idempotence \eqref{Eq:Idempotence} and semisymmetry \eqref{Eq:SemiSymm}, it is only natural to ask if there is a quasigroup variety whose linear theory rests on $X^2+X+1$.  Forcing a CML-linear quasigroup to be semisymmetric requires $L=R^{-1}$, and $R^3+1=0$.  Further imposing idempotence leads to $R^2-R+1=0$.  To produce $R^2+R+1=0,$ we would like the automorphism to be not order $6,$ but rather order $3,$ as $R^3-1=(R-1)(R^2+R+1),$ whence the following definition.       

\begin{definition}\label{Def:EisQgps}
The variety of \emph{right Eisenstein quasigroups}, denoted $\mathbf{RE}$, is generated by \eqref{IL}-\eqref{SR} and
\begin{align}
    (yx\cdot x)x&=y. \label{Eq:Order3RDiv}
\end{align}
The variety $\mathbf{LE}$ of \emph{left Eisenstein quasigroups} is defined by \eqref{IL}-\eqref{SR}, and 
\begin{align}
    x(x\cdot xy)=y. \label{Eq:Order3LDiv}
\end{align}
\end{definition}

\begin{prop}\label{Prop:EisQgps}
Suppose $\emph{Lin}(M, R, L)$ is a CML-linear pique.  Then this structure belongs to $\mathbf{RE}$ if and only if $R^2+R+1=0$, and $\emph{Lin}(M, R, L)$ belongs to $\mathbf{LE}$ if and only if $L^2+L+1=0$.
\end{prop}

\begin{proof}
The linearization of $(yx\cdot x)x$ is 
\begin{equation}\label{Eq:LinOrder3Div}
    yR^3+xLR^2+xLR+xL=yR^3+xL(R^2+R+1).
\end{equation}
If \eqref{Eq:Order3RDiv} holds in $\text{Lin}(M, R, L)$, then substituting $y=0$ means that $0=xL(R^2+R+1)$ for all $x\in M$.  Since $L$ is an automorphism, $R^2+R+1=0$.  Conversely, if we take a CML-linear pique $\text{Lin}(M, R, L)$, and we assume $R^2+R+1=0,$ then $R^3-1=(R-1)(R^2+R+1)=0$, so we find that $yR^3+xL(R^2+R+1)=yR^3=y.$  Proof of the characterization for CML-linear piques in $\mathbf{LE}$ follows by an identical argument; just switch the roles of $L$ and $R$. 
\end{proof}

\begin{remark}
Although \eqref{Eq:Order3RDiv} guarantees $R^2+R+1=0$ in a CML-linear pique $\text{Lin}(M, R, L)$, it says nothing of $L$.  Hence, a Theorem \ref{Th:CatRelat}-type result for $\mathbb{Z}\otimes\mathbf{RE}$ would relate $\mathbf{RE}$ not to the Eisenstein integers, but rather to \\ $\mathbb{Z}[\zeta]\otimes_{\mathbb{Z}}\mathbb{Z}[L, L^{-1}].$  A search for conditions on quasigroup words furnishing some $p(R)\in \mathbb{Z}[R, R^{-1}]$ so that $L=p(R)$ --and thus producing a functor $\mathbb{Z}[\zeta]$-$\mathbf{Mod} \to \mathbb{Z}\otimes\mathbf{V}$, where $\mathbf{V}$ is a subvariety of $\mathbf{RE}$-- is reserved for future research.   
\end{remark}

Our final result involves the notion of \emph{quasigroup homotopy}.  Let $(Q, \cdot)$ and $(Q^\prime, *)$ be quasigroups.  A triple of functions $(f_1, f_2, f_3)$ mapping $Q$ into $Q^\prime$ is a homotopy if for all $x, y\in Q,$ $(x)f_1*(y)f_2=(x\cdot y)f_3$.  If each of the $f_i$ are bijections, then $(f_1, f_2, f_3)$ is an \emph{isotopy}.  An isotopy of the form $(f_1, f_2, 1_Q):(Q, \cdot)\to (Q, *)$ is \emph{principal} \cite{Sade}.  Homotopy is a vital aspect of nonassociative geometry \cite{Sabinin}.  Specifically, every $3$-web is coordinatized by $\mathbf{Qtp}(\top, Q)$, the set of quasigroup homotopies from a singleton into some quasigroup $Q$.  Isotopes of $Q$ yield different coordinatizations of the same $3$-web (cf. \cite[Th.~1]{QHSRA}).    

\begin{thm}\label{Th:MTSIsotopes}
Every distributive Mendelsohn quasigroup is principally isotopic to a left Eisenstein quasigroup.
\end{thm}

\begin{proof}
The case of the empty quasigroup is true by the fact that this structure may be viewed as both distributive and left Eisenstein.  Let $\text{Lin}(M, R, 1-R)$ be a CML-linear pique which is Mendelsohn; by Theorem \ref{Th:B&S}, this is the same thing as fixing an arbitrary nonempty, distributive Mendelsohn quasigroup.  We claim $\text{Lin}(M, -R, R^2)$ lies in $\mathbf{LE},$ and $(-1, -1, 1)$ is a principal isotopy from $\text{Lin}(M, -R, R^2)$ onto the CML-linear pique $\text{Lin}(M, R, 1-R)$.
Note that $(R^2)^2+R^2+1=R^3(R)+R^2+1=-R+R^2+1=0,$ so $\text{Lin}(M, -R, R^2)$ is in $\mathbf{LE}$ by Proposition \ref{Prop:EisQgps}.  Furthermore, 
\begin{align*}
(x\cdot y)1_M&=-xR+yR^2 \\
&=-xR+y(R-1) \\
&=-xR+(-y)(1-R)\\
&=x(-1_M)*y(-1_M),
\end{align*}
verifies $(-1, -1, 1):\text{Lin}(M, -R, R^2)\to \text{Lin}(R, 1-R)$ is an isotopy.  
\end{proof}

\section*{Acknowledgements}
The author would like to thank his doctoral advisor, Jonathan D. H. Smith, for helping him to prepare the manuscript, and for pointing out the connection between the ramified classification problem of Section \ref{Sec:Ramified} and congruence subgroups.

\end{document}